\documentclass[preprint,10pt]{elsarticle}

\usepackage{amsthm,amsmath,amsfonts,amssymb,amscd,mathrsfs}
\usepackage{txfonts}
\usepackage{supertabular,soul}
\usepackage[usenames,dvipsnames]{xcolor}
\usepackage{tikz, graphicx,color,geometry}
 \usepackage{multirow}
 \usetikzlibrary{arrows}
\usepackage[pdftex,
            pdfauthor={Barrett},
            pdftitle={Automorphisms, Equitable Partitions, and Spectral Graph Theory},
            pdfsubject={Equitable Partitions, Spectral Graph Theory},
            pdfkeywords={Equitable Partitions, Spectral Graph Theory}]{hyperref}

\usepackage{bbm} 
\usepackage{hyperref}
\usepackage{yfonts}
\usepackage{eucal}
\usepackage{overpic}
\usetikzlibrary{calc}
\usepackage{enumitem}

\newcommand{\comment}[1]{}

\DeclareMathOperator{\diag}{diag}
\DeclareMathOperator{\rank}{rank}

\newcommand{\defital}{\textit}

\newcommand{\C}{\mathbb C}

\newcommand{\Aut}{\text{Aut}}
\newcommand{\sL}{\mathscr{L}}
\newcommand{\cT}{\mathcal{T}}

\newtheorem{thm}{Theorem}[section]

\newtheorem{lem}[thm]{Lemma}

\newtheorem{prop}[thm]{Proposition}

\newtheorem{cor}[thm]{Corollary}

\theoremstyle{definition}
\newtheorem{remark}[thm]{Remark}

\newtheorem{defn}[thm]{Definition}

\newtheorem{example}[thm]{Example}
\newtheorem*{examplestar}{Example}
\theoremstyle{remark}

\providecommand*{\propertyautorefname}{Property}

\let\oldmarginpar\marginpar
\renewcommand\marginpar[1]{\oldmarginpar[\raggedleft\footnotesize #1]%
{\raggedright\footnotesize #1}}

\begin{document}
\begin{frontmatter}

\date{\today}

\title{Equitable Decompositions of Graphs with Symmetries}
\author[wb]{Wayne Barrett}
\address[wb]{Department of Mathematics, Brigham Young University, Provo, UT 84602, USA, wb@mathematics.byu.edu}

\author[amanda]{Amanda Francis}
\address[amanda]{Department of Mathematics, Computer Science and Engineering, Carroll College, Helena, MT 59601, USA, afrancis@carroll.edu}
\author[ben]{Benjamin Webb}
\address[ben]{Department of Mathematics, Brigham Young University, Provo, UT 84602, USA, bwebb@math.byu.edu}


\begin{abstract}

We investigate connections between the symmetries (automorphisms) of a graph and its spectral properties. {Whenever} a graph has a symmetry, i.e. a nontrivial automorphism $\phi$, it is possible to use $\phi$ to decompose any matrix $M\in\mathbb{C}^{n \times n}$ appropriately associated with the graph. The result of this decomposition is a number of strictly smaller matrices whose collective eigenvalues are the same as the eigenvalues of the original matrix $M$. Some of the matrices that can be decomposed are the graph's adjaceny matrix, Laplacian matrix, etc. Because this decomposition has connections to the theory of equitable partitions it is referred to as an \emph{equitable decomposition}. Since the graph structure of many real-world networks is quite large and has a high degree of symmetry, we discuss how equitable decompositions can be used to effectively bound both the network's spectral radius and spectral gap, which are associated with dynamic processes on the network. Moreover, we show that the techniques used to equitably decompose a graph can be used to bound the number of simple eigenvalues of undirected graphs, where we obtain sharp results of Petersdorf-Sachs type.
\end{abstract}


\end{frontmatter}

\section{Introduction}

In spectral graph theory one studies the relationship between two kinds of objects, a graph $G$ (which for us may be directed or undirected) and an associated matrix $M$.  The major aims of spectral graph theory are to determine information about a graph by examining the eigenvalues of $M$ and inferring information about the eigenvalues from the graph structure.

The particular type of structure we consider in this paper is the notion of graph symmetries, which can be understood via graph automorphisms.  Formally, a \emph{graph automorphism} of $G$ is defined to be a permutation $\phi: V(G) \to V(G)$ of the graph's vertices $V(G)$ that preserves adjacencies. More intuitively, a graph automorphism describes how parts of a graph can be permuted in a way that preserves the graph's overall structure.  In this sense these \emph{parts}, i.e., subgraphs, are symmetrical and together constitute a graph symmetry.
An appealing result of this type is the following about eigenvalue multiplicity (\cite{Cvet}, p. 82).

\begin{thm}[Moshowitz; Petersdorf and Sachs]
If $G$ is a simple graph with an automorphism of order greater than 2, then the adjacency matrix for $G$ has a multiple eigenvalue.
\end{thm}

 One limitation of this theorem is that it does not describe how many multiple eigenvalues there can be. Another is that it treats only the adjacency matrix, even though the same conclusion applies to many matrices associated with $G$.

Another connection between eigenvalues and automorphisms is seen in the theory of equitable partitions.  An equitable partition of the vertex set of a graph $G$ typically arises from an automorphism of $G$ and yields a matrix of reduced size all of whose eigenvalues are eigenvalues of the adjacency matrix of $G$.  This attractive concept has proved very useful in studying highly symmetrical graphs, yet it only yields \emph{some} of the eigenvalues, and again it seems that it should apply to a larger collection of matrices.

In this paper we show that for a large class of matrices associated with a graph $G$ and a large class of graph automorphisms, we can decompose the matrix via similarity into a direct sum of smaller matrices and thus the collective eigenvalues of these smaller matrices are the same as the eigenvalues of the original matrices (see Theorem \ref{thm:1} and \ref{thm:2}).  Since one of the summands in this direct sum is the same matrix that one finds via an equitable partition, we refer to the full decomposition as an \emph{equitable decomposition}.  As an interesting comparison, we recall that the spectral decomposition of a matrix requires knowledge of all eigenvalues and eigenvectors, while our decomposition can be found from an associated automorphism (symmetry) alone.


In the case of an undirected graph our main results allow us to give a sharp upper bound on the number of simple eigenvalues of many matrices associated with $G$.  This idea is not new as a very strong result of this type can be found in \cite{PetersdorfSachs}.  But again, that result is only stated for the adjacency matrix and its generality prevents an accompanying sharpness result.

We use the term \emph{automorphism  compatible} for the class of matrices to which our results apply.  This means that for an $n \times n$ matrix $M = [m_{ij}]$ associated with a graph $G$ on $n$ vertices and any automorphism $\phi$ of $G$, $m_{\phi(i),\phi(j)} = m_{ij}$ for all $i,j \in \{1, 2, \ldots, n\}$. Matrices falling into this class are the adjacency matrix, Laplacian matrix, signless Laplacian, normalized Laplacian, distance matrix, weighted Laplacians and some analogues for directed graphs (see Proposition \ref{prop:ACmatrices}). Note that this definition is stronger than is necessary.  If a matrix $M$ is compatible with any particular automorphism of $G$, then it can be decomposed using that  automorphism.

To prove our main result, we begin by considering automorphisms of a graph for which all orbits are of the same size; we denote these as \emph{uniform automorphisms}.  We consider this case first for two reasons: First, many of the most well-known graphs have a non-trivial uniform automorphism and second,  the general theory of equitable decompositions can be built from this class of autmorphisms.

We then extend these results to \emph{basic automorphisms}, those automorphisms whose orbits are all of size 1 or $k$, for some positive integer $k>1$. Such automorphisms can be viewed as a \emph{local symmetry} of a graph. We show that one surprising consequence of knowing a basic automorphism of a graph is that one can quickly calculate a subset of its associated eigenvalues using the theory of equitable decompositions developed here. That is, one can use local information regarding a graph's structure to determine properties of the graph's eigenvalues, which in general depend on the entire graph structure! We explain how to obtain all eigenvalues from \emph{any} automorphism in a separate paper. The key idea is that for any automorphism $\phi$, some power of $\phi$ is a basic automorphism.

This method of using local symmetries to find eigenvalues of a graph is perhaps most useful in analyzing the spectral properties of real-world networks. These networks are typically large and have a high degree of symmetry when compared, for instance, to randomly generated graphs \cite{MacArthur}. From a practical point of view, the size of these networks limit our ability to quickly compute their associated eigenvalues. However, their large degree of symmetry suggests that finding network symmetries is much more feasible and can be used to decompose the graph into more manageable pieces.

In Section \ref{sec:App} we describe how this method can be used to bound the spectral radius {and spectral gap (algebraic connectivity) associated with a network (graph), both of} which are important in the study of networks. For instance, the spectral radius can be used to study stability properties of a network related to its dynamics \cite{Bunimovich}. Similarly, the spectral gap is used to determine a number of dynamic properties on networks including synchronization thresholds and the rate of convergence to synchronization and consensus \cite{Pikovsky,Almendral,Atay}. If the matrix associated with a network is stochastic then the spectral gap is related to mixing times of the associated Markov chain, and to first-passage times of random walks on the network \cite{Donetti}.

The remainder of the paper is organized as follows. In Section \ref{sec:EP} we give the background, including key definitions regarding the class of graphs we consider in this paper.  We also review the basic theory of equitable partitions and give our first result which extends the theory of equitable partitions. We build our results to include both uniform and basic automorphisms in Sections \ref{sec:EqDecomp} and \ref{sec:basic}, respectively.
Section \ref{sec:App} describes how this theory can be used to estimate key properties of a network, i.e. large graphs with symmetries. We conclude in Section \ref{sec:conclusion} with some closing remarks including a few open questions regarding equitable decompositions.

\section{Graphs and Equitable Partitions}\label{sec:EP}

The main type of mathematical objects we consider in this paper are graphs. A \emph{graph} $G$ is made up of a finite set of vertices $V(G)=\{1,\dots,n\}$ and and a finite set of edges $E(G)$. The vertices of a graph are typically represented by points in the plane and an edge by a line or curve in the plane that connects two vertices. A graph can be \emph{undirected}, meaning that each edge $\{i,j\}\in E$ can be thought of as an unordered set or a multiset if $i=j$ ($\{i,i\}\in E$). For $\{i,j\}\in E$ we say that this edge is \emph{incident to} both vertex $i$ and vertex $j$ and if two vertices are connected by an edge then these vertices are said to be \emph{adjacent}. A graph is \emph{directed} when each edge is {\emph{directed}, in which case $(i,j)$ is an ordered tuple. In both a directed and undirected graph, a} \emph{loop} is an edge with only one vertex ($\{i, i\}\in E$).

The graphs we will consider in this paper are those that are either directed or undirected. Throughout the paper we let $G$ denote either a directed or undirected graph. It should be clear from the context whether we are considering a directed or undirected graph, with or without loops, etc. We also allow weighted graphs (both directed and undirected), that is, graphs in which each edge ( $\{i,j\}$ or $(i,j)$) is assigned a numerical weight ($w(i,j)$). See Proposition \ref{prop:ACmatrices} and Remark \ref{rem:WAmatrices}.

There are a number of matrices that can be associated with a given graph $G$ (see Section \ref{sec:EqDecomp}). Two of the most common, which we will use as examples throughout the paper are the adjacency matrix $A(G)$ of a graph and the Laplacian matrix $L(G)$ if $G$ is \emph{simple graph}, i.e. an undirected graph without loops. The \emph{adjacency matrix} $A=A(G)$ of the graph $G$ is the $0$-$1$ matrix given by
\[
a_{ij}=
\begin{cases}
1 &\text{if} \ \ \{i,j\}\in E(G)\\
0 &\text{otherwise}.
\end{cases}
\]
To define the Laplacian matrix of a simple graph let $D_G=\diag[\deg(1),\dots,\deg(n)]$ denote the \emph{degree matrix} of $G$. Then the \emph{Laplacian matrix} $L(G)$ of $G$ is the matrix $L(G)= D_G - A(G)$. In the case of a weighted graph, the \textit{weighted adjacency matrix} is given by the edge weights,
\[
a_{ij} =
\begin{cases}
w(i,j) &\text{if} \ \ \{i,j\}\in E(G)\\
0 &\text{otherwise}.
\end{cases}
\]

We let $\sigma(M)$ denote the \emph{eigenvalues} of the $n\times n$ matrix $M$. For us $\sigma(M)$ is a multiset with each eigenvalue in $\sigma(M)$ listed according to its multiplicity, where an eigenvalue is called \emph{simple} if it has multiplicity 1.

An \textit{equitable partition} of a graph's vertices $V(G)$ can be used to find a number of the eigenvalues of a matrix $M$ associated with a graph $G$. Specifically, an equitable partition generates a matrix $M_{\pi}$ whose eigenvalues are contained in $\sigma(M)$ (\cite{Godsil} \S 9.3, \cite{Cvet} \S 3.9). The matrix $M_{\pi}$ is constructed as follows.



\begin{defn}\label{def:EQ} \textbf{(Equitable Partition)}
Given a  graph $G$, and a matrix $M = [m_{i j}]$ associated with $G$,  a partition $\pi$ of $V(G)$, $V(G) = V_1 \cup \ldots \cup V_k$ is \defital{equitable} with respect to $G$ and $M$, if for all $i$, $j \in \{1, 2, \ldots, k\}$
\begin{equation}\label{eq:1}
\sum_{t \in V_j} m_{st} = d_{ij}
\end{equation}
is a constant $d_{ij}$ for any  $s \in V_i$.

The $k \times k$ matrix $M_\pi = d_{ij}$ is called the \defital{divisor matrix} of $M$ associated with the partition $\pi$. 
\end{defn}

{Definition \ref{def:EQ} is, in fact, an extension of the\emph{ standard definition} of an equitable partition, which is defined for \emph{simple graphs}, i.e. unweighted undirected graphs without loops. For such graphs the requirement that $\pi$ be an equitable partition is equivalent to the condition that any vertex $\ell \in V_i$ has the same number of neighbors in $V_j$ for all $i,j \in \{1, \ldots, k\}$ (for example, see p. 195-6 of  \cite{Godsil}). An equitable partition of a simple graph is given in the following example.}

\begin{example}\label{ex:1}

Let $G$ be the simple graph shown in Figure \ref{fig:ex1}. An equitable partition $\pi$ of $V(G)$ associated with the adjacency matrix $A=A(G)$ is
\[
V_1 = \{1,3,5,7\}, \ \ V_2 = \{ 2, 4, 6, 8 \}.
\]
The reason for this, following Definition \ref{def:EQ}, is that $d_{11} = 0, d_{12} = 2, d_{21} = 2, d_{22} = 3$.  The divisor matrix of $A$ associated with $\pi$ is then
\[
A_\pi = \left[\begin{array}{rr} 0 & 2 \\ 2 & 3 \end{array}\right]
\]
with eigenvalues $\sigma(A_\pi)=\{4,-1\}$, which are two of the eigenvalues of the original matrix $A$, for which
\[
\sigma(A)=\{4,1,1,0,-1,-1,-2,-2\}.
\]
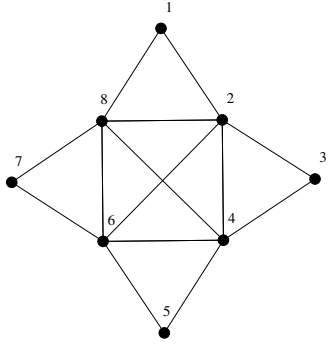
\begin{figure}
\begin{center}

\scalebox{.8}{
\begin{tikzpicture}[line cap=round,line join=round,>=triangle 45,x=1.0cm,y=1.0cm]
\draw (1.66,3.96) -- (1.68,1.96) -- (3.68,1.98) -- (3.66,3.98) -- cycle;
\draw (3.66,3.98) -- (3.68,1.98) -- (5.2020508075688785,2.9973205080756875) -- cycle;
\draw (1.66,3.96) -- (3.66,3.98) -- (2.642679491924312,5.502050807568878) -- cycle;
\draw (1.68,1.96) -- (1.66,3.96) -- (0.162050807568877486,2.942679491924312) -- cycle;
\draw (3.68,1.98) -- (1.68,1.96) -- (2.6973205080756877,0.43794919243112234) -- cycle;
\draw (1.66,3.96)-- (3.68,1.98);
\draw (3.66,3.98)-- (1.68,1.96);
\begin{scriptsize}
\draw  [fill=black]  (1.66,3.96) circle (2.5pt);
\draw  (1.7,4.32) node {$8$};
\draw  [fill=black] (1.68,1.96) circle (2.5pt);
\draw (1.82,2.32) node {$6$};
\draw  [fill=black] (3.68,1.98) circle (2.5pt);
\draw (3.82,2.34) node {$4$};
\draw  [fill=black] (3.66,3.98) circle (2.5pt);
\draw (3.8,4.34) node {$2$};
\draw [fill=black] (5.2020508075688785,2.9973205080756875) circle (2.5pt);
\draw (5.34,3.36) node {$3$};
\draw  [fill=black] (2.642679491924312,5.502050807568878) circle (2.5pt);
\draw (2.78,5.86) node {$1$};
\draw [fill=black] (0.162050807568877486,2.942679491924312) circle (2.5pt);
\draw (0.28,3.3) node {$7$};
\draw  [fill=black] (2.6973205080756877,0.43794919243112234) circle (2.5pt);
\draw (2.74,0.8) node {$5$};
\end{scriptsize}
\end{tikzpicture}
}
\end{center}

\caption{The graph $G$ considered in Example \ref{ex:1}, whose adjacency matrix has the equitable decomposition $V(G)=V_1\cup V_2$ where $V_1=\{1,3,5,7\}$ and $V_2=\{2,4,6,8\}$.}
 \label{fig:ex1}
\end{figure}
Alternatively, for the same graph and its Laplacian matrix $L=L(G)$, one can easily check that
\[
L_\pi = \left[\begin{array}{rr} 2 & -2 \\ -2 & 2 \end{array}\right].
\]
Again, the eigenvalues $\sigma(L)=\{0,4\}$ are among the eigenvalues of the Laplacian matrix
\[
\sigma(L)=\{0,4-\sqrt{6},4-\sqrt{6},2,4,6,4+\sqrt{6},4+\sqrt{6}\}.
\]
\end{example}

Containment of the eigenvalues $\sigma(A_\pi)\subset\sigma(A)$ and $\sigma(L_\pi)\subset\sigma(L)$ for the matrices in Example \ref{ex:1} is a consequence of a simple generalization of the result (in \cite{Godsil} Theorem 9.3.3, \cite{Cvet} Theorem 3.9.5), which is only stated for the adjacency matrix.

\begin{thm}\label{thm:EQsimple}
Let $G$ be a graph and $M$ a matrix associated with $G$. Suppose the partition $\pi$
\[V(G) = V_1 \cup \ldots \cup V_k\]
is equitable, and let $M_\pi$ be the associated divisor matrix.  Then $\sigma(M_\pi) \subseteq \sigma(M)$.
\end{thm}

The reason equitable partitions are of interest is that they allow us a way of computing some spectral properties of a matrix based on structural properties of an associated graph. However, this information is incomplete. A natural question is whether this missing information can be recovered in the same way the eigenvalues of an equitable partition are found. That is, for a matrix $M$ can we find other matrices $B_1,\dots,B_k$ besides $M_{\pi}$ whose collective eigenvalues give \emph{all} the eigenvalues the original matrix $M$?

With this in mind, the following lemma generalizes a well-known formula for the eigenvalues of a circulant matrix (see \cite{HornJohn}, p. 100). This lemma  will be useful in developing a method for finding all the remaining eigenvalues of the matrix $M$ for a certain class of equitable partitions.


\begin{lem}\label{lem:1}
Let $C\in\mathbb{C}^{n \times n}$ be the block-circulant matrix
\[
C = \left[ \begin{array}{lllll}
C_0 & C_1 & C_2 & \ldots & C_{k-1} \\
C_{k-1} & C_0 & C_1 & \ldots & C_{k-2}\\
C_{k-2} & C_{k-1} & C_0 & \ddots & \vdots\\
\vdots & \vdots & \ddots & \ddots & C_1 \\
C_1 & C_2 & \ldots & C_{k-1} & C_0
\end{array}\right], \text{ where each block is } r \times r,
\]
and let
\[
S = \left[\begin{array}{lllll}
I & I & I & \ldots & I \\
I & \omega I & \omega^2 I & \ldots & \omega^{k-1} I \\
I & \omega^2 I & \omega^4 I & \ldots & \omega^{2(k-1)} I \\
\vdots & \vdots & \vdots & & \vdots \\
I & \omega^{k-1} I & \omega^{2(k-1)} I & \ldots & \omega^{(k-1)^2} I
 \end{array}\right],
\]
where $\omega = e^{2\pi i /k}$, and $S$ is partitioned conformally with $C$. Then
\[
S^{-1} C S = B_0 \oplus B_1 \oplus \ldots \oplus B_{k-1},
\]
where
\[
B_j = \sum_{m = 0}^{k-1}\omega^{jm}C_m, \quad j = 0, 1, \ldots, k-1.
\]
Consequently $\sigma(C) = \sigma(B_0) \cup \sigma(B_1) \cup \ldots \cup \sigma(B_{k-1})$.
\end{lem}

\begin{proof}
The matrix $S$ has orthogonal columns, so $S^{-1} = (1/k)S^*$ where $*$ denoted the conjugate transpose.

Then
\[
(S^{-1}CS)_{pq} = \frac{1}{k} \sum_{j=1}^k \sum_{l=1}^k \bar{\omega}^{(p-1)(j-1)}C_{l-j} \omega^{(l-1)(q-1)},
\]
where $l-j$ is interpreted modulo $k$.

Upon the substitution $m = l-j$, this becomes
\[
(S^{-1} C S)_{pq} = \frac{1}{k}\sum_{j=1}^k \sum_{m=0}^{k-1} \bar{\omega}^{(p-1)(j-1)}C_{m} \omega^{(m+j-1)(q-1)}
= \frac{1}{k}\sum_{m=0}^{k-1}\omega^{(m-1)(q-1)+(p-1)}C_{m} \sum_{j=1}^k  {\omega}^{j(q-p)},
\]
which is 0 if $p \neq q$.  If $p = q$, this reduces to
\[
(S^{-1} C S)_{pp} = \frac{1}{k}\sum_{m=0}^{k-1}\omega^{m(p-1)}C_{m} = B_{p-1}.
\]
Therefore,
\[
S^{-1} C S = B_0 \oplus B_1 \oplus \ldots \oplus B_{k-1}
\]
and $\sigma(C)=\sigma(B_0)\cup \sigma(B_1) \cup \cdots \cup \sigma(B_{k-1})$.

\end{proof}


To see how Lemma \ref{lem:1} is related to equitable partitions we apply this lemma to the matrices considered in  Example \ref{ex:1}.
\vspace{.2cm}

\noindent \textbf{Example \ref{ex:1} continued}.
The adjacency matrix for the graph $G$ can be expressed as the block circulant matrix
\[
A = \left[\begin{array}{rrrr}
A_0 & A_1 & A_2 & A_3 \\
A_3 & A_0 & A_1 & A_2 \\
A_2 & A_3 & A_0 & A_1 \\
A_1 & A_2 & A_3 & A_0
\end{array}\right],
\]
\[\text{with} \ A_0 = \left[\begin{array}{rr} 0 & 1 \\ 1 & 0 \end{array}\right], \ A_1 = \left[\begin{array}{rr} 0 & 0 \\ 1 & 1 \end{array}\right], \ A_2= \left[\begin{array}{rr} 0 & 0 \\ 0 & 1 \end{array}\right], \ \text{and} \ A_3 = \left[\begin{array}{rr} 0 & 1 \\ 0 & 1 \end{array}\right].
\]
Then for $\omega = i$, the matrices $B_0, B_1, B_2, B_3$ in Lemma \ref{lem:1} are given by
\[
B_j = \left[\begin{array}{rr} 0 & 1 \\ 1 & 0 \end{array}\right] +
i^j \left[\begin{array}{rr} 0 & 0 \\ 1 & 1 \end{array}\right]+
(-1)^j \left[\begin{array}{rr} 0 & 0 \\ 0 & 1 \end{array}\right]+
(-i)^j  \left[\begin{array}{rr} 0 & 1 \\ 0 & 1 \end{array}\right],
\]
which gives us
\[
B_0 = \left[\begin{array}{rr} 0 & 2 \\ 2 & 3 \end{array}\right], \ B_1 =  \left[\begin{array}{cc} 0 & 1-i \\ 1+i & -1 \end{array}\right], \ B_2 = \left[\begin{array}{rr} 0 &0 \\ 0 & -1 \end{array}\right], \ B_3 =  \left[\begin{array}{cc} 0 & 1+i \\ 1-i & -1 \end{array}\right].
\]
We recognize that $B_0$ is the divisor matrix associated with $G$ and $A$. We already know that $\sigma(B_0) = \{4,-1\}$ but can compute that $\sigma(B_1) = \sigma(B_3) = \{1,-2\}$ and $\sigma(B_2) = \{0,-1\}$. Hence
\[
\sigma(A)=\sigma(B_0)\cup\sigma(B_1)\cup\sigma(B_2)\cup\sigma(B_3) = \{4,1,1,0,-1,-1,-2,-2\},
\]
 by Lemma \ref{lem:1}.

If we carry out this calculation for the Laplacian matrix $L$ in Example \ref{ex:1}, we obtain the matrices
\[
B_0 = \left[\begin{array}{rr} 2 & -2 \\ -2 & 2 \end{array}\right], \ B_1 =  \left[\begin{array}{cc} 2 & -1+i \\ -1-i & 6 \end{array}\right], \ B_2 = \left[\begin{array}{rr} 2 &0 \\ 0 & 6 \end{array}\right], \ B_3 =  \left[\begin{array}{cc} 2 & -1-i \\ -1+i & 6 \end{array}\right].
\]
For these matrices $\sigma(B_0) = \{0,4\}$, $\sigma(B_1) = \sigma(B_3) = \{4-\sqrt{6}, 4+\sqrt{6}\}$ and $\sigma(B_2) = \{2, 6\}$, and again
\[
\sigma(L) = \sigma(B_0)\cup\sigma(B_1)\cup\sigma(B_2)\cup\sigma(B_3).
\]

The key reason it is possible to decompose the matrices $A$ and $L$ into smaller matrices while preserving their eigenvalues is the symmetry present in the graph $G$. In the following section we describe how a matrix that has a structure that mimics the structure of a graph can be decomposed with respect to the graph's symmetries. As in Example \ref{ex:1}, the result is a number of smaller matrices whose collective eigenvalues are the same as the eigenvalues of the original matrix.

\section{Equitable Decompositions}\label{sec:EqDecomp}
The lemma and example of the previous section point to a general method that can be applied to find all the eigenvalues of matrices associated with a particular kind of graph. The following definitions and notation are needed to state the result.

\begin{defn}
An \emph{automorphism} $\phi$ of a (weighted or unweighted) graph $G$ is a permutation of $V(G)$ such that the adjacency matrix $A$ satisfies $a_{ij} = a_{\phi(i) \phi(j)}$ for each pair of vertices $i$ and $j$.
Notice that in the case of an unweighted graph, this is equivalent to saying
$i$ and $j$ are adjacent in $G$ if and only if $\phi(i)$ and $\phi(j)$ are adjacent in $G$. The set of all automorphisms of $G$ is a group, denoted by $\Aut(G)$.
\end{defn}

Importantly, a graph's group of automorphisms characterizes the symmetries in the graph's structure. For a graph $G$ with automorphism $\phi$, we define the relation $\sim$ on $V(G)$ by $u \sim v$ if and only if $v = \phi^j(u)$ for some nonnegative integer $j$. It follows that $\sim$ is an equivalence relation on $V(G)$, and the equivalence classes are called the \emph{orbits} of $\phi$.

In Example \ref{ex:1} in the previous section, $\phi=(1, 3, 5, 7)(2, 4, 6, 8)$ is an automorphism of the graph $G$ and $V_1 = \{1, 3, 5, 7 \}$, $V_2 = \{2,4 , 6, 8\}$ are the corresponding orbits.  As we already saw $V(G)=V_1\cup V_2$ gives an equitable partition of $V(G)$.  This is no accident and illustrates the following basic fact (see \cite{Godsil}, p. 196).
\begin{prop}\label{prop:1}
Let $\phi$ be an automorphism of $G$.  Then the orbits of $\phi$ give an equitable partition of $V(G)$.
\end{prop}
\noindent It is worth mentioning that the converse of proposition \ref{prop:1} is not true (see \cite{ChanGodsil}, p. 84).

 To use this connection between automorphisms and equitable partitions we consider those matrices associated with a graph $G$ whose structure mimics the structure, in particular the symmetries, of $G$.

\begin{defn}\label{def:autocomp}
Let $G$ be a graph on $n$ vertices. An $n \times n$ matrix $M = [m_{ij}]$ is \emph{automorphism compatible} on $G$ if, given any automorphism $\phi$ of $G$ and any $i, j \in \{1, 2, \ldots, n\}$,
$m_{\phi(i) \phi(j)} = m_{i j}$.
\end{defn}

Some of the most well-known matrices that are associated with a graph are automorphism compatible.

\begin{prop}\label{prop:ACmatrices}
Given an undirected graph $G$, its adjacency matrix, combinatorial Laplacian matrix, signless Laplacian matrix, normalized Laplacian matrix, and distance matrix are all automorphism compatible. Also, the weighted adjacency matrix of a weighted graph is automorphism compatible.

\end{prop}

\begin{proof}
Recall the definitions:
\begin{itemize}
\item The Laplacian matrix of $G$ is $L(G)= D_G - A(G)$,
\item The signless Laplacian matrix of $G$ is $Q(G)= D_G + A(G)$,
\item The normalized Laplacian matrix of $G$ is ${\mathcal L}(G) = D_G^{-1/2} L(G) D_G^{-1/2}$.
\end{itemize}
For the normalized Laplacian matrix we assume also that every vertex of $G$ has positive degree.

{Since every vertex in an orbit of an automorphism must have the same degree, the matrix $D_G$ is automorphism compatible. {Furthermore, if $A=A(G)$ it then} follows that $aD_G + bA$ is automorphism compatible for all $a, b \in \C$.  This includes the Laplacian and signless Laplacian matrix of $G$.  Now consider the normalized Laplacian matrix $\mathcal{L}=\mathcal{L}(G)$, which can be written as
\[\sL = I - D_G^{-\frac{1}{2}}AD_G^{-\frac{1}{2}}.\]
The identity matrix is clearly automorphism compatible and}
\[
{\left( D_G^{-\frac{1}{2}}AD_G^{-\frac{1}{2}} \right)_{\phi(i) \phi(j)}
=d_{\phi(i)  }^{-\frac{1}{2}}a_{\phi(i) \phi(j)}d_{  \phi(j)}^{-\frac{1}{2}}
=d_{i }^{-\frac{1}{2}}a_{i j}d_{j  }^{-\frac{1}{2}}
=\left( D_G^{-\frac{1}{2}}AD_G^{-\frac{1}{2}} \right)_{i j},}
\]
where $d_{i} = deg(i)$.
{It follows that the normalized Laplacian matrix is likewise automorphism compatible.}

Recall that the \emph{distance} $d(i,j)$ in $G$ is the length of the shortest path in $G$ from $i$ to $j$. The \emph{distance matrix} $\mathcal D=\mathcal{D}(G)$ is the $n \times n$ matrix whose $i,j$ entry is $d_{ij}=d(i,j)$.
The definition of a graph automorphism implies that for any two vertices $i$, $j$ in a graph, {the distance} $d(\phi(i),\phi(j)) = d(i,j)$. Therefore, the distance matrix is also automorphism compatible.
\end{proof}

\begin{remark}\label{rem:WAmatrices}
The Laplacian matrix, signless Laplacian matrix, normalized Laplacian matrix, and distance matrix can all be interpreted as weighted adjacency matrices for the weighted graphs their entries suggest. In fact, the notions of weighted adjacency matrix and automorphism compatible matrix are interchangeable; we can interpret any automorphism compatible matrix as a weighted adjacency matrix of the appropriate weighted graph.
\end{remark}

We next give the simplest version of our main result, in which we consider graphs that have an automorphism with all orbits of the same size. The reason we present this special case first is because it is both cleaner and because it applies to so many of the most well-known graphs.

\begin{defn}\label{def:UniAut} \textbf{(Uniform Automorphisms)}
An automorphism $\phi$ of a graph $G$ has uniform orbit size if every orbit in $\phi$ has the same cardinality.  We call such an automorphism a \emph{uniform automorphism} and this common cardinality is its \emph{size}.
\end{defn}

An important consequence of having a uniform automorphism $\phi$ is that either $\phi$ has size $k=1$ and all vertices of $G$ are fixed by $\phi$ or $k>1$ and every vertex belongs to a orbit of size $k$ that includes other vertices.

\begin{defn}\label{def:transversal}
Let $\phi$ be a uniform automorphism of the graph $G$ on $n$ vertices of size $k>1$.  Choose one vertex from each orbit, and let $\mathcal{T}$ be the set of these chosen vertices.  We say that $\cT$ is a \emph{transversal} of the orbits of $\phi$.  Further we define the set
\[
\cT_\ell = \{\phi^\ell(v) \ | \ v \in \cT\}
\]
for $\ell = 0,1, \ldots, k-1$ to be the $\ell$th power of $\cT$. If $k=1$ then $\phi=id$ is trivial and $\cT=V(G)$.
\end{defn}

\begin{examplestar}
For the graph $G$ considered in example \ref{ex:1}, $\cT_0 = \{1, 2\}$ is a transversal with $\cT_1 = \{3,4\}$, $\cT_2 = \{5,6\}$, and $\cT_3 = \{7,8\}$ its powers.
\end{examplestar}

Labeling the vertices of a graph $G$ by the powers of one of its transversals is key to being able to apply Lemma \ref{lem:1} to an automorphism compatible matrix on $G$. The notion of a transversal consequently allows us to decompose any matrix compatible with $G$ with respect to any of its uniform automorphisms.

\begin{thm}\label{thm:1} \textbf{(Uniform Equitable Decompositions)}
Let $G$ be a graph on $n$ vertices, let $\phi$ be a uniform automorphism of $G$ of size $k$, let $\cT_0$ be a transversal of the orbits of $\phi$, and let $M$ be an automorphism compatible matrix  on $G$.  Set $M_\ell = M[\cT_0, \cT_\ell]$, $\ell = 0, 1, \ldots, k-1$, let $\omega=e^{2 \pi i /k}$, and let
\[
B_j = \sum_{\ell=0}^{k-1} \omega^{j\ell}M_\ell, \ \ j = 0, 1, \ldots, k-1.
\]
Then
\begin{equation}\label{eq:spectrum}\sigma(C) = \sigma(B_0) \cup \ldots \sigma(B_{k-1}),\end{equation}
and $B_0$ is the divisor matrix of $M$ associated with the partition given by the $n/k$ orbits of $\phi$.

\end{thm}

\begin{proof}
We may assume that the vertices of $G$ and the rows and columns of $M$ are labeled in the order $\cT_0, \ldots, \cT_{k-1}$.  Since $M$ is automorphism compatible,
$M[\cT_s, \cT_t] =M[\phi(\cT_s), \phi(\cT_t)] = M[\cT_{s+1}, \cT_{t+1}]$ for all $s, t \in \{0, 1, \ldots, k-1\}$ (Note $\phi(\cT_{k-1}) = \cT_0$), and $M$ is a block circulant matrix.
Then (\ref{eq:spectrum}) follows from Lemma \ref{lem:1}.

Since $B_0 = \sum_{\ell=0}^{k-1} M_\ell$, letting $V_j$ be the $j$th orbit, we know that
\[
(B_0)_{ij} = \sum_{\ell=0}^{k-1} (M_\ell)_{ij} = \sum_{\ell=0}^{k-1}(M[\cT_0, \cT_\ell])_{ij} = \sum_{\ell=0}^{k-1}m_{i \phi^\ell(j)} = \sum_{\ell \in V_j} m_{i\ell} = d_{ij}
\]
by Definition \ref{def:EQ}. This completes the proof.
\end{proof}

Recall from Lemma \ref{lem:1} that
\begin{equation}\label{eq:matrixdecomp}
S^{-1}MS = B_0 \oplus B_1 \oplus \cdots \oplus B_{k-1}.
\end{equation}

Since $B_0$ is the divisor matrix induced by the orbits of an automorphism, it seems natural to call (\ref{eq:matrixdecomp}) an \textit{equitable decomposition} of the matrix $M$ (or of the graph $G$ if $M=A(G)$ is its adjacency matrix). The major difference between and equitable decomposition and an equitable partition is that, instead of preserving a few eigenvalues via an equitable partition, an equitable decomposition maintains all the eigenvalues of the given matrix.

An important distinction between the usual diagonalization of a matrix and the block diagonalization given in equation \eqref{eq:matrixdecomp} is that the latter requires no knowledge of the eigenvalues or eigenvectors of $M$, only of the automorphism $\phi$. That is, once a uniform automorphism of a graph is known, we can immediately decompose any automorphism compatible matrix $M$ without finding any other spectral quantities. Still, the resulting matrices will have the same eigenvalues as the original matrix $M$.

Although Theorem \ref{thm:1} holds for both directed and undirected graphs, the existence of symmetries in an undirected graph has some important consequences for the number of simple eigenvalues any automorphism compatible matrix can have.

\begin{cor}\label{cor:simpleevals}
Let $G$ be a graph on $n$ vertices with a uniform automorphism $\phi$ of size $k >1$, let $M$ be a symmetric automorphism compatible matrix of $G$, and let $r = n/k$.  Then
\begin{enumerate}
\item\label{item1} If $k$ is odd, there are at most $r$ simple eigenvalues of $G$.
\item\label{item2} If $k$ is even, there are at most $2r$ simple eigenvalues of $G$.
\item\label{item3} The bounds in parts \ref{item1} and \ref{item2} are sharp in the sense that equality holds for infinitely many graphs $G$.
\end{enumerate}
\end{cor}

\begin{remark}
Statements \ref{item1} and \ref{item2} are essentially corollaries of the main result of Petersdorf and Sachs \cite{PetersdorfSachs}, although their result is stated for only the adjacency matrix of a graph.  This result is also given as Theorem 5.7 in \cite{Cvet95}, p. 140, which is the following theorem.
\end{remark}

\begin{thm}[Petersdorf, Sachs]
Let $G$ be a multigraph, let $\phi \in \Aut(G)$, and let $\alpha(\phi)$ be the number of odd cycles and $\beta(G)$ be the number of even cycles of $\phi$.  Then $G$ has at most $\alpha(\phi) + 2 \beta(\phi)$ simple eigenvalues.
\end{thm}

If this result is extended to automorphism compatible matrices, parts \ref{item1} and \ref{item2} of Corollary \ref{cor:simpleevals} are both direct consequences. Here, we give a straightforward proof of Corollary \ref{cor:simpleevals}.

\begin{proof}
Since $M$ is symmetric with block circulant structure, we have $(M_{k-j})^T = M_j$ from which it follows that $B_{k-j}^T = B_j$. Then the eigenvalues of $B_j$ and $B_{k-j}$ are identical for each $j = 1, \ldots, k-1$.  For $k$ odd, the only possible simple eigenvalues are those of $B_0$ and for $k$ even, the only possible simple eigenvalues are those of $B_0$ and $B_{k/2}$.

To prove part \ref{item3}, consider the $k$-sun shown in Figure \ref{fig:ksun}, the simple graph on $n = 2k$ vertices with edge set
\[
 \{\{1,2\}, \{3,4\}, \ldots, \{2k-1,2k\},\{2,4\}, \{4,6\}, \ldots, \{2k-2, 2k\}, \{2k,2\}\}.
 \]

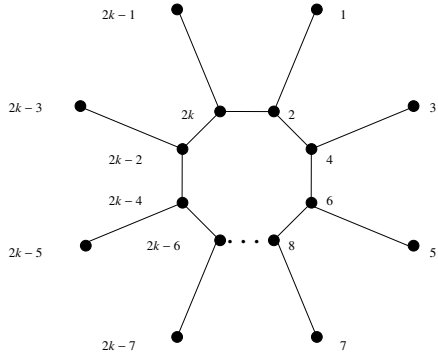
\begin{figure}
\begin{center}

\begin{overpic}[scale=.45]{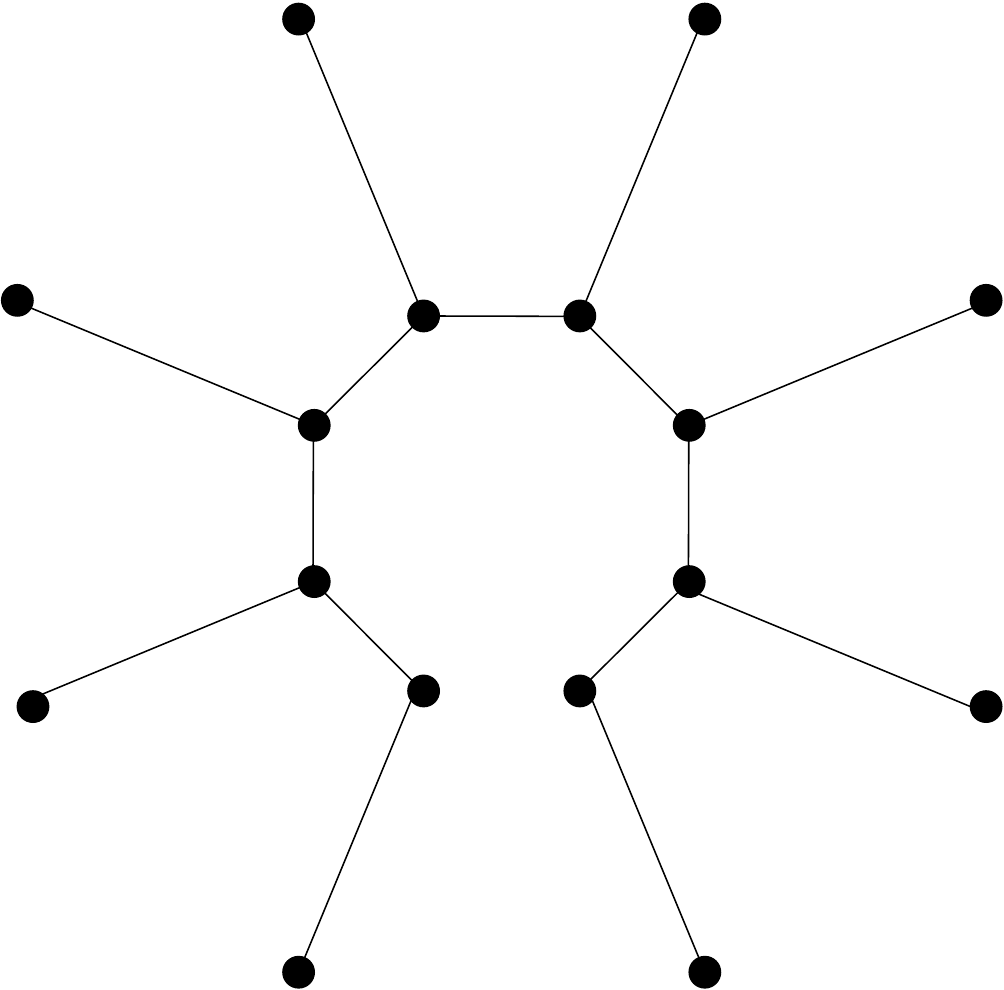}
    \put(77,94){\tiny$1$}
    \put(8,94){\tiny$2k-1$}
    \put(62,65){\tiny$2$}
    \put(31,65){\tiny$2k$}
    \put(103,67){\tiny$3$}
    \put(-19,67){\tiny$2k-3$}
    \put(73,52){\tiny$4$}
    \put(10,52){\tiny$2k-2$}
    \put(103,25){\tiny$5$}
    \put(-19,25){\tiny$2k-5$}
    \put(73,40){\tiny$6$}
    \put(10,40){\tiny$2k-4$}
    \put(77,-2){\tiny$7$}
    \put(8,-2){\tiny$2k-7$}
    \put(62,27){\tiny$8$}
    \put(21,27){\tiny$2k-6$}
    \put(44,29){\large{$\dots$}}
    \end{overpic}\end{center}
\caption{{The $k$-sun on $n = 2k$ vertices is shown, which has the automorphism $\phi= (1,3,5, \ldots, 2k-1)(2,4,6, \ldots, 2k)$.}}\label{fig:ksun}
\end{figure}

Let $\phi$ be the automorphism
\[
\phi = (1,3,5, \ldots, 2k-1)(2,4,6, \ldots, 2k)
\]
so that $r = 2$.  Applying Theorem \ref{thm:1} to the adjacency matrix $A$ of the $k$-sun, we have
\[
B_j = \left[\begin{array}{cc}
0 & 1 \\ 1 & 2\cos(2 \pi j /k)\end{array}\right], \ j = 0, 1, \ldots, k-1,
\]
with characteristic polynomials $p_j(t) = p_{B_j}(t) = t^2 - 2\cos(2\pi j / k)t-1$.
In particular the characteristic polynomial of $B_0$ is $p_0(t) = t^2 - 2t -1$.
\vspace{.2in}

\noindent \textbf{ Case 1:} $k$ is odd\\
We know that the eigenvalues of $B_j$, $j = 1, \ldots, k-1$ are multiple eigenvalues of $A$.  Calculating the resultant of $p_0(t)$ and $p_j(t)$, $1 \leq j \leq k-1$, we have
\[
\text{Res}(p_0, p_j) = \left[ \begin{array}{rrcc}
1 & 0 & 1 & 0 \\
-2 & 1 & -2 \cos(2\pi j /k) &1 \\
-1 & -2 & -1 & -2\cos(2 \pi j / k ) \\
0 & -1 & 0 & -1
\end{array}\right] = -4 [1 - \cos(2\pi j /k)]^2 < 0.
\]

\noindent {Hence,} $p_0$ and $p_j$ have no roots in common, which implies that the $r=2$ roots of $p_0$ are simple eigenvalues of $A$. Therefore, part \ref{item3} of Corollary \ref{cor:simpleevals} is sharp in this case.
\vspace{.2in}

\noindent \textbf{Case 2: } $k$ is even\\
We know that the only possible simple eigenvalues of $A$ are those of $B_0$ and $B_{k/2}$. Here, the respective characteristic polynomials are
\[
p_0(t) = t^2 - 2t -1 \text{ and } p_{k/2}(t) = t^2 + 2t-1.
\]
The resultant found in the case that $k$ \emph{is odd} again verifies that the roots of $p_0(t)$ are distinct from those of all the other $p_j(t)$, and a similar calculation verifies that the roots of $p_{k/2}(t)$ are distinct from those of $p_j(t)$ for $j \neq 0, k/2$.  It {then} follows that $A$ has $2r =4$ simple eigenvalues and that the bounds in parts \ref{item1} and \ref{item2} of Corollary \ref{cor:simpleevals} are sharp for the adjacency matrix of the $k$-sun, for every $k$.
\end{proof}

It is worth noting that Corollary \ref{cor:simpleevals} may be sharp for one automorphism compatible matrix associated with a graph and not for another. For example, note that the graph in Example \ref{ex:1} has the automorphism $\phi = (1, 3, 5, 7)(2, 4, 6, 8)$, for which $r = 8/4 = 2$. Corollary \ref{cor:simpleevals} guarantees that any automorphism compatible matrix can have at most four simple eigenvalues. In this case the the graph's Laplacian matrix does have four simple eigenvalues, but its adjacency matrix only has two.

In the following example we show that Corollary \ref{cor:simpleevals} may not hold if $M$ is not symmetric.

\begin{example}\label{ex:directed}
{Consider the directed graph $K$ shown in Figure \ref{fig:ladder}, which resembles a ladder graph but does not have a symmetric adjacency matrix $A=A(K)$. The graph has $n=2k$ vertices for any $k>1$ and, as can be seen, $\phi=(1,3,5,\dots, 2k-1)(2,4,6,\dots, 2k)$ is an automorphism of $K$, so that $r=n/k=2$.}

\begin{figure}
%
%
%
    \begin{center}
   \begin{tikzpicture}[line cap=round,line join=round,x=1.0cm,y=1.0cm,scale = 1]
\coordinate (x1)  at (0,1);
\coordinate (x2)  at (0,-1);
\coordinate (x3) at (1.5,1);
 \coordinate (x4)at (1.5,-1) ;
 \coordinate (x5) at (3,1);
 \coordinate (x6) at (3,-1);
 \coordinate (x5p) at (4.5,1);
 \coordinate (x6p) at (4.5,-1);
 \coordinate (x7m) at (5.2,1);
 \coordinate (x8m) at (5.2,-1);
 \coordinate (x7) at (6.7,1);
 \coordinate (x8) at (6.7,-1);
 \coordinate (xa) at (-.1,1.5);
 \coordinate (xb) at (6.7,1.5);
 \coordinate (xc) at (-.1,-1.5);
 \coordinate (xd) at (6.7,-1.5);

\draw [->](x1) -- ($(x3)+(-.1,0)$)  ;
\draw [->] (x3) -- ($(x5)+(-.1,0)$) ;
\draw [->] (x5)--($(x5p)+(-.1,0)$) ;
\draw [->](x2) -- ($(x4)+(-.1,0)$) ;
\draw [->](x4) -- ($(x6)+(-.1,0)$) ;
\draw [->] (x6)--($(x6p)+(-.1,0)$) ;
\draw [->](x7m)--($(x7)+(-.1,0)$);
\draw [->](x8m) -- ($(x8)+(-.1,0)$) ;
\draw [->](x2) -- ($(x1)+(0,-.1)$) ;
\draw [->](x4) -- ($(x3)+(0,-.1)$) ;
\draw [->](x6) -- ($(x5)+(0,-.1)$) ;
\draw [->](x8) -- ($(x7)+(0,-.1)$) ;

\draw (xb) -- (xa)  ;
\draw (xd) -- (xc)  ;

\draw[->] (xa) arc (90:270:.25);
\draw[->] (xc) arc (270:90:.25);
\draw (xd) arc (270:450:.25);
\draw (xb) arc (450:270:.25);

\begin{scriptsize}
\draw [fill=black] (x1) circle (2.5pt);
\draw [fill=black] (x2) circle (2.5pt);
\draw [fill=black] (x3) circle (2.5pt);
\draw [fill=black] (x4) circle (2.5pt);
\draw [fill=black] (x5) circle (2.5pt);
\draw [fill=black] (x6) circle (2.5pt);
\draw [fill=black] (x7) circle (2.5pt);
\draw [fill=black] (x8) circle (2.5pt);
\node at ($(x1)+(0,.25)$) {$1$};
\node at ($(x2)+(0,-.25)$) {$2$};
\node at ($(x3)+(0,.25)$) {$3$};
\node at ($(x4)+(0,-.25)$) {$4$};
\node at ($(x5)+(0,.25)$) {$5$};
\node at ($(x6)+(0,-.25)$) {$6$};
\node at ($(x7)+(0,.25)$) {$7$};
\node at ($(x8)+(0,-.25)$) {$8$};
\node at (4.8,1) {$\ldots$};
\node at (4.8,-1) {$\ldots$};
\node at (3.75,-2) {$K$};
\end{scriptsize}
\end{tikzpicture}
\end{center}
\caption{{The directed graph $K$, considered in example \ref{ex:directed}, on $2k$ vertives with automorphism $\phi=(1,3,5,\dots,2k-1)(2,4,6,\dots,2k)$ is shown. If $k$ is odd then the adjacency matrix of $A=A(K)$ has $2k$ simple eigenvalues. If $k$ is even then $A$ has $2k-2$ simple eigenvalues.}}\label{fig:ladder}
\end{figure}
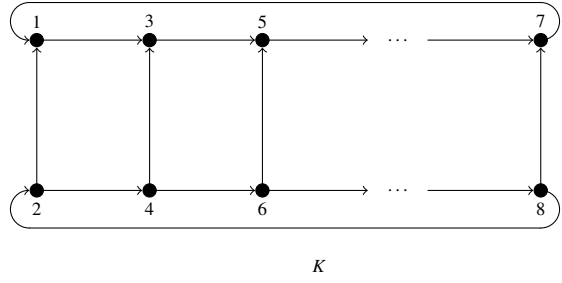

{Using theorem \ref{thm:1} on $A$, the matrix $B_j$ corresponding to the automorphism $\phi$ is
\[
B_j=
\left[\begin{array}{cc}
\omega^j&1\\
1&\omega^j
\end{array}\right] \ \ \text{for} \ \ j=0,\dots,k-1
\]
where $\omega=e^{2\pi i/k}$. Since $\sigma(B_j)
=\{\pm1+\omega^j\}$ then the matrix $A$ has eigenvalues

\[
\sigma(A)=\{\pm1+e^{2\pi i j/k}: j=0,1,\dots,k-1\}.
\]
Hence, if $k$ is odd, then $2k$ eigenvalues of $A$ are simple. If $k$ is even, all eigenvalues of $L$ are simple except $0$, which has multiplicity 2.}

Note that the graph $K$ has $n=2k$ vertices implying $r=n/k=2$ for the automorphism $\phi$.
For $k >2$, there are greater than $r$ simple eigenvalues, demonstrating that the conclusion in  Corollary \ref{cor:simpleevals} need not hold if $M$ is not symmetric.
\end{example}

\section{Basic Automorphisms}\label{sec:basic}

In this section we extend the results of Lemma \ref{lem:1} and Theorem \ref{thm:1} to include a slightly broader class of automorphisms which we will call basic. A \emph{basic automorphism} is a generalization of a uniform automorphism, in the sense that it contains orbits of some uniform size $k>1$ and may also contain orbits of size 1.

To extend our results of the previous section to basic automorphisms we require the following lemma, which is a generalization of Lemma \ref{lem:1}.

\begin{lem}\label{lem:2}
Let $C$ be the block matrix
\begin{equation}\label{eq:circulant}
C = \left[\begin{array}{llllll}
F & H & H & H & \cdots & H \\
L & C_0 & C_1 & C_2 & \cdots & C_{k-1} \\
L & C_{k-1} & C_0 & C_1 & \cdots & C_{k-2} \\
L & C_{k-2} & C_{k-1} & C_0 & \cdots & C_{k-3} \\
\vdots & \vdots & \vdots & \vdots & & \vdots \\
L & C_1 & C_2 & C_3 & \cdots & C_0 \\
\end{array}\right],
\end{equation}
where $F$ is $p \times p$, $H$ is $p \times r$, $L$ is $r \times p$, and each $C_j$ is $r \times r$.  Let $T = I_p \oplus S$ where $S$ is the matrix given in the statement of Lemma \ref{lem:1}.  Then
\[
T^{-1} C T = \left[\begin{array}{rr} F & kH \\ L & B_0 \end{array}\right]
\oplus B_1 \oplus B_2 \oplus \cdots \oplus B_{k-1},
\]
where
\[
B_j = \sum_{m=0}^{k-1} \omega^{jm}C_m \ \ \text{for} \ \ j = 0, 1, \ldots, k-1.
\]
Consequently,
\[
\sigma(C) = \sigma\left(\left[\begin{array}{rr} F & kH \\ L & B_0 \end{array}\right] \right)
\cup \sigma(B_1) \cup \sigma(B_2) \cup \cdots \cup \sigma(B_{k-1}).
\]
\end{lem}

\begin{proof}
Let $R$ be the matrix in the statement of Lemma \ref{lem:1}.  Then with $P = \left[ H \quad H \cdots H\right]$ and
$Q = [ L^T \quad L^T \cdots L^T]^T$,
\[
T^{-1}CT = \left[\begin{array}{rr} I_p & 0 \\ 0 & (1/k) S^*\end{array}\right]
\left[\begin{array}{rr} F & P \\ Q & R \end{array}\right]
\left[\begin{array}{rr} I_p & 0 \\ 0 & S\end{array}\right]
=\left[\begin{array}{rr} F & PS \\ (1/k)S^*Q & (1/k) S^* R S \end{array}\right].
\]
Since $PS = [kH \quad 0 \quad \cdots \quad 0 ]$ and
$S^*Q = [kL^T \ 0 \ \cdots \ 0 ]^T$, we have by Lemma \ref{lem:1}
\[
T^{-1}CT = \left[\begin{array}{rrrrrrr}
F & kH & 0 & 0 & \cdots & \cdots & 0 \\
L & B_0 & 0 & 0 & \cdots & \cdots & 0 \\
0 & 0 & B_1 & 0 & \cdots & \cdots & 0 \\
0 & 0 & 0 & B_2  & 0  & \cdots & 0 \\
\vdots & \vdots & \vdots & 0 & \ddots & \ddots & \vdots \\
\vdots & \vdots & \vdots & \vdots & \ddots & \ddots & \vdots \\
0 & 0 & 0 &0  & \cdots  & 0 &B_{k-1} \\
\end{array}\right],
\]
and the result follows.
\end{proof}

In order to state an analogue of Theorem \ref{thm:1} we need to extend definitions \ref{def:UniAut} and \ref{def:transversal}.

\begin{defn} \textbf{(Basic Automorphism)}
If $\phi$ is an automorphism of a graph $G$ with orbits of size $k >1$ or 1, we call $\phi$ a \textit{basic automorphism} of $G$ with orbit size $k$. The vertices with orbit size 1 are said to be \emph{fixed} by $\phi$.
\end{defn}

\begin{defn}\label{def:semitrans}
Given a basic automorphism with orbit size $k$, choose one vertex from each orbit of size $k$, and let $\cT_0$ be the set of these chosen vertices.  We call $\cT_0$ a \textit{semi-transversal} of the orbits of $\phi$.  We define the powers $\cT_m$ as before.
\end{defn}

The reason $\cT_0$ is not a transversal in Definition \ref{def:semitrans} is that we do not include vertices from those orbits of $\phi$ of size one. The notions of a basic automorphism and semi-transversal allow us to extend the result of Theorem \ref{thm:1} as follows.

\begin{thm}\label{thm:2} \textbf{(Basic Equitable Decomposition)}
Let $G$ be a graph on $n$ vertices, let $\phi$ be a basic automorphism of $G$ of size $k>1$, let $\cT_0$ be a semi-transversal of the $k$-orbits of $\phi$, let $\cT_f$ be the vertices fixed by $\phi$, let $p = |\cT_f|$, and let $M$ be an automorphism compatible matrix on $G$.  Set $F = M[\cT_f,\cT_f]$, $H = M[\cT_f,\cT_0]$, $L=M[\cT_0,\cT_f]$, $M_m = M[\cT_0, \cT_m]$, for $m = 0, 1, \ldots, k-1$, $\omega = e^{2 \pi i /k}$, and
\[
B_j = \sum_{m=0}^{k-1} \omega^{jm} M_m,  \ \ j = 0, 1, \ldots, k-1.
\]
Then
\begin{equation}\label{eq:spectrum2}
\sigma(M) = \sigma\left(\left[\begin{array}{rr} F & kH \\ L & B_0 \end{array}\right] \right)
\cup \sigma(B_1) \cup \sigma(B_2) \cup \cdots \cup \sigma(B_{k-1}),\end{equation}
and $\left[\begin{array}{rr} F & kH \\ L & B_0 \end{array}\right] $ is the divisor matrix of $M$ in the equitable partition associated with $\phi$.
\end{thm}

\begin{proof}
We may assume that the vertices of $G$ and rows and columns of $M$ are labeled in the order $\cT_f, \cT_0, \cT_1, \ldots, \cT_{k-1}$.  Since the vertices of $\cT_f$ are fixed by $\phi$ and $\phi$ is automorphism compatible,
\[H = M[\cT_f, \cT_0] = M[\phi^m(\cT_f), \phi^m(\cT_0)] = M[\cT_f, \cT_m] \ \ \text{for} \ \ m = 0, 1, \ldots, k-1.
\]
Similarly,
$L = M[\cT_m, \cT_f]$ for $m = 0, 1, \ldots, k-1$.

As in the proof of Theorem \ref{thm:1}, $M[\cT_s, \cT_t] = M[\cT_{s+1},\cT_{t+1}]$ for all $s,t \in \{0, \ldots, k-1\}$. Then $M$ has the form of the matrix in \eqref{eq:circulant}, so \eqref{eq:spectrum2} follows by Lemma \ref{lem:2}.

To prove the last statement, we wish to show that the divisor matrix $M_\phi$ is the same as the ``first'' matrix in the decomposition $\tilde{B} = \left[\begin{array}{rr} F & kH \\ L & B_0\end{array}\right]$. Since the vertices are labeled in the order of fixed points and then semi-transversals, we have the following orbits of vertices,
\[
\begin{array}{ll}
V_1 = \{1\}, \ldots V_p = \{p\}, \ V_{p+1} = \{p+1, p+1+(n-p)/k, \ldots, p+1+(k-1)(n-p)/k\}, \ldots,\hspace{1cm} \\
\hspace*{4cm} V_{p+(n-p)/k} = \{p + (n-p)/k, p+2(n-p)/k, \ldots, n = (p + k(n-p)/k)\}.
\end{array}
\]
Recall that $(M_\phi)_{ij} = \sum_{t \in V_j} m_{st}$ for any $s \in V_i$.  We have four cases to consider.
First,if $V_i$ and $V_j$  are fixed point orbits ($i, j \leq p$) then clearly
\[
(M_\phi)_{ij} =  m_{ij} = (F)_{ij} = (\tilde{B})_{ij}.
\]
If $ i \leq p$ and $j >p$, we have
\[
(M_\phi)_{ij} = \sum_{t \in V_j} m_{it} =\sum_{\ell=0}^{k-1} m_{i\phi^{\ell}(j)}
=\sum_{\ell=0}^{k-1} m_{\phi^{\ell}(i)\phi^{\ell}(j)}  =\sum_{\ell=0}^{k-1} m_{ij} = k (H)_{i,j-p}   = (\tilde{B})_{ij}.
\]
If $ i>  p$ and $j\leq p$, we have
\[
(M_\phi)_{ij} =  m_{ij} = (L)_{i-p,j}  = (\tilde{B})_{ij}.
\]
Finally, if both $i$ and $j$ are greater than $p$, we have
\[
(M_\phi)_{ij} =  \sum_{t \in V_j} m_{it} =\sum_{\ell=0}^{k-1} m_{i\phi^{\ell}(j)}  =
\sum_{\ell=0}^{k-1} M[\mathcal{T}_0, \mathcal{T}_\ell]_{i-p \ j-p}  = (B_0)_{i-p, j-p} = (\tilde{B})_{ij}.
\]

\end{proof}

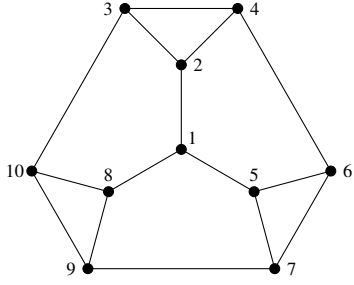
\begin{figure}\label{fig:trigraph}
\begin{center}
\begin{tikzpicture}[line cap=round,line join=round,>=triangle 45,x=1.0cm,y=1.0cm,scale = .75]
\coordinate (x1)  at (0,0);
\coordinate (x2)  at (0,1.5);
\coordinate (x3) at (-1,2.5);
 \coordinate (x4)at (1,2.5) ;
 \coordinate (x5) at (1.29,-.75);
 \coordinate (x6) at (2.65,-.384);
 \coordinate (x7) at (1.656,-2.116);
 \coordinate (x8) at (-1.29,-.75);
 \coordinate (x9) at (-1.656,-2.116);
 \coordinate (x10) at (-2.65,-.384);

\draw (x1) -- (x2) -- (x3) -- (x4) -- (x2);
\draw (x1) -- (x8) -- (x9) -- (x10) -- (x8);
\draw (x1) -- (x5) -- (x6) -- (x7) -- (x5);
\draw (x4) -- (x6) ;
\draw (x3) -- (x10) ;
\draw (x9) -- (x7) ;

\begin{scriptsize}
\draw [fill=black] (x1) circle (2.5pt);
\draw [fill=black] (x2) circle (2.5pt);
\draw [fill=black] (x3) circle (2.5pt);
\draw [fill=black] (x4) circle (2.5pt);
\draw [fill=black] (x5) circle (2.5pt);
\draw [fill=black] (x6) circle (2.5pt);
\draw [fill=black] (x7) circle (2.5pt);
\draw [fill=black] (x8) circle (2.5pt);
\draw [fill=black] (x9) circle (2.5pt);
\draw [fill=black] (x10) circle (2.5pt);
\node at ($(x1)+(.2,.2)$) {$1$};
\node at ($(x2)+(.3,0)$) {$2$};
\node at ($(x3)+(-.3,0)$) {$3$};
\node at ($(x4)+(.3,0)$) {$4$};
\node at ($(x5)+(0,.3)$) {$5$};
\node at ($(x6)+(.3,0)$) {$6$};
\node at ($(x7)+(.3,0)$) {$7$};
\node at ($(x8)+(0,.3)$) {$8$};
\node at ($(x9)+(-.3,0)$) {$9$};
\node at ($(x10)+(-.3,0)$) {$10$};
\end{scriptsize}
\end{tikzpicture}
\end{center}
\caption{The graph $G$ considered in Example \ref{ex:basic} along with the basic automorphism $\phi = (1)(2,5,8)(3,6,9)(4,7,10)$.}
\end{figure}

\begin{example}\label{ex:basic}
Let $G$ be the undirected graph shown in Figure \ref{fig:trigraph} and let $\phi$ be the automorphism given by $\phi = (1)(2,5,8)(3,6,9)(4,7,10)$. The adjacency matrix $A=A(G)$ has the form
\[
A = \left[\begin{array}{rrrr}
F & H  & H & H \\
L & A_0 & A_1 & A_2 \\
L & A_2 & A_0 & A_1 \\
L & A_1 & A_2 & A_0\\
\end{array}\right]
\]
with $F = [0]$, $H = [1 \ 0 \ 0 ]$, $L = H^T$,
\[
A_0 = \left[\begin{array}{rrr} 0 & 1 & 1 \\ 1 & 0 & 1 \\ 1 & 1 & 0 \end{array}\right], \ \
A_1 = \left[\begin{array}{rrr} 0 & 0 & 0 \\ 0 & 0 & 0 \\ 0 & 1 & 0 \end{array}\right], \ \
A_2 = C_1^T.
\]
Then using Theorem \ref{thm:2} we have
\[
B_0 = \left[\begin{array}{rrr} 0 & 1 & 1 \\ 1 & 0 & 2 \\ 1 & 2 & 0 \end{array}\right], \ \
B_1 = \left[\begin{array}{ccc} 0 & 1 & 1 \\ 1 & 0 &1 + \omega^2 \\ 1 & 1 + \omega & 0 \end{array}\right], \ \
B_2 = B_1^T
\]
and
\[
\left[\begin{array}{rr} F & kH \\ L & B_0\end{array}\right]
= \left[\begin{array}{rrrr} 0 & 3 & 0 & 0 \\ 1 & 0 & 1 & 1 \\ 0 & 1 & 0 & 2 \\ 0 & 1 & 2 & 0 \end{array}\right].
\]
The eigenvalues of this last matrix are $\{3, 1, -2, -2\}$ while the eigenvalues of $B_1$ are approximately $\sigma(B_1)\approx \{1.87939, -.347296, -1.53209\}$.  By Theorem \ref{thm:2}, the eigenvalues of $A$ are then approximately
\[
\sigma(A)\approx \{3, 1.87939,1.87939,1, -.347296,-.347296, -1.53209,-1.53209,-2,-2\}.
\]
\end{example}

Since not every automorphism is either uniform or basic it is important to note that, if $G$ has a nonuniform automorphism $\phi$ then some power of $\phi$ is a basic automorphism. That is,  it is always possible given an automorphism $\phi$ of a graph to either use this automorphism to decompose an associated matrix or to find some power of $\phi$ that can be used to decompose this matrix.
This fact can be used recursively to \textit{fully} decompose a matrix, and can be interpreted as fully decomposing a graph (see \cite{BFSW}).

Corollary \ref{cor:simpleevals} can also be extended to basic automorphisms as follows.

\begin{cor}\label{cor:basicevals}
Let $G$ be a graph on $n$ vertices with a basic automorphism $\phi$ having $N$ orbits of size 1 and all other orbits of size $k>1$ and let $M$ be a symmetric automorphism compatible matrix of $G$. Then
\begin{enumerate}
\item\label{item1} If $k$ is odd, $A$ has at most $N+(n-N)/k$ simple eigenvalues.
\item\label{item2} If $k$ is even, $A$ has at most $N+2(n-N)/k$ simple eigenvalues.
\item\label{item3} The {bound in part \ref{item1} is sharp in that it holds for infinitely many $N$ and infinitely many $k$.}
\end{enumerate}
\end{cor}

The proof of Corollary \ref{cor:basicevals} is similar to that of Corollary \ref{cor:simpleevals}.  Again, the simple eigenvalues can only come from the matrices $M_\phi $   (and $B_{k/2}$ if $k$ is even) in the matrix decomposition.

{The following examples verifies statement \ref{item3}. (It is currently unknown if the bound in \ref{item2} is sharp for infinitely many $N$ and $k$.)}
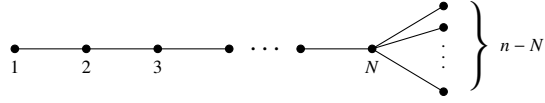
\begin{figure}[h]
    \begin{center}
   \begin{tikzpicture}[line cap=round,line join=round,x=1.0cm,y=1.0cm,scale = .95]
\coordinate (x1)  at (-3,0);
\coordinate (x2)  at (-2,0);
\coordinate (x3) at (-1,0);
 \coordinate (x4)at (0,0) ;
 \coordinate (x5) at (1,0);
 \coordinate (x6) at (2,0);
 \coordinate (x7) at (3,.6);
 \coordinate (x8) at (3,.3);
 \coordinate (x9) at (3,-.6);

\draw (x1) -- (x2) --(x3) -- (x4) ;
\draw  (x5) -- (x6) --(x7);
\draw  (x6)--(x8) ;
\draw (x6) -- (x9) ;

\begin{scriptsize}
\draw [fill=black] (x1) circle (1.5pt);
\draw [fill=black] (x2) circle (1.5pt);
\draw [fill=black] (x3) circle (1.5pt);
\draw [fill=black] (x4) circle (1.5pt);
\draw [fill=black] (x5) circle (1.5pt);
\draw [fill=black] (x6) circle (1.5pt);
\draw [fill=black] (x7) circle (1.5pt);
\draw [fill=black] (x8) circle (1.5pt);
\draw [fill=black] (x9) circle (1.5pt);
\node at ($(x1)+(0,-.25)$) {$1$};
\node at ($(x2)+(0,-.25)$) {$2$};
\node at ($(x3)+(0,-.25)$) {$3$};
\node[scale=1.5] at ($(x4)+(.5,0)$) {$\ldots$};
\node at ($(x6)+(0,-.25)$) {$N$};
\node at ($(x8)+(0,-.3)$) {$\vdots$};
\node[scale=2] at ($(x9)+(.5,.65)$) {$\Bigg\}$};
\node at ($(x9)+(1.1,.65)$) {$n-N$};
\end{scriptsize}
\end{tikzpicture}
\end{center}
\caption{{The tree $T$, considered in Example \ref{ex:tree}, on $n$ vertives with automorphism $\phi=(N+1, N+2, \ldots, n)$ is shown. If $n$ is even and $N$ is odd then the adjacency matrix of $A=A(T)$ has $N+1$ simple eigenvalues. }}\label{fig:tree}
\end{figure}

\begin{example}\label{ex:tree}

Let $n$ be a positive even integer and let $N < n-2$ be odd.  Let $T$ be the tree in Figure \ref{fig:tree} with automorphism  $\phi = (N+1, N+2, \ldots, n)$.  Then $k = n-N$ and the bound in 1 is
 \[
 N +\frac{n-N}{n-N} = N+1.
 \]
 Let $M$ be the adjacency matrix of $T$.  Then
 \[
 M = \left[\begin{array}{rr}
 A(P_N) & B \\ B^T & 0\end{array}\right[
 \]
 where $B$ is the $N \times n-N$ matrix with every entry in its last row equal to 1 and all other entries equal to 0. The leading $N+1 \times N+1$ principal submatrix of $M$ is $A(P_{N+1})$, which is a tridiagonal matrix with 1's on the subdiagonal and superdiagonal and all other entries 0.  Since $\det A(P_{N+1})$ is $\pm1$ for all odd $N$, $\rank M \geq N+1$.  But the last $k$ rows and columns of $M$ are the same, so $\rank M \leq N+1$.  Thus we have $\rank M = N+1$.

 Suppose $\lambda$ is a nonzero eigenvalue of $M$.  Let $R$ be the matrix obtained from $M - \lambda I$ by deleting the $N$th row and first column.  Then $R$ is a lower triangular matrix with all diagonal entries nonzero.  (There are $N-1$ entries equal to 1 followed by $n-N$ entries equal to $-\lambda$.)  Thus, $R$ is invertible and we have
 \[
 n-1 \geq \rank (M-\lambda I) \geq \rank R = n-1.
 \]
 Then $\rank (M - \lambda I) = n-1$, and $\lambda$ is a simple eigenvalue of $M$.  Since the number of nonzero eigenvalues of $M$ is $\rank M$, we have $N+1$ simple eigenvalues of $M$, which matches the bound in 1, proving that it is sharp.
\end{example}


\section{Some Applications}\label{sec:App}
Suppose $M$ is a matrix associated with a graph $G$. The eigenvalues of $M$ are \emph{global} characteristics of the matrix $M$ in the sense that they depend, in general, on all entries of the matrix or equivalently on the entire structure of the graph $G$. In contrast, the symmetries of a graph $G$ are inherently \emph{local} when they correspond to basic automorphisms.

A particularly important example is in the case of finding the eigenvalues associated with the graph structure of a real network. Reasons for this include the fact that most real networks are large, often having either thousands, hundreds of thousands, or more vertices \cite{MacArthur}. Second, real networks are on average much more structured and in particular have more symmetries than random graphs (see \cite{MacArthur}). Third, there is often only partial or local information regarding the structure of many of these networks because of the complications in obtaining network data.

The implication, with respect to equitable decompositions, is that by finding a number of graph symmetries it is possible to gain information regarding the graph's set of eigenvalues. This information, although incomplete, can be used to determine spectral properties of the network as is demonstrated in the following example.

\begin{figure}
\begin{center}

\begin{overpic}[scale=.44]{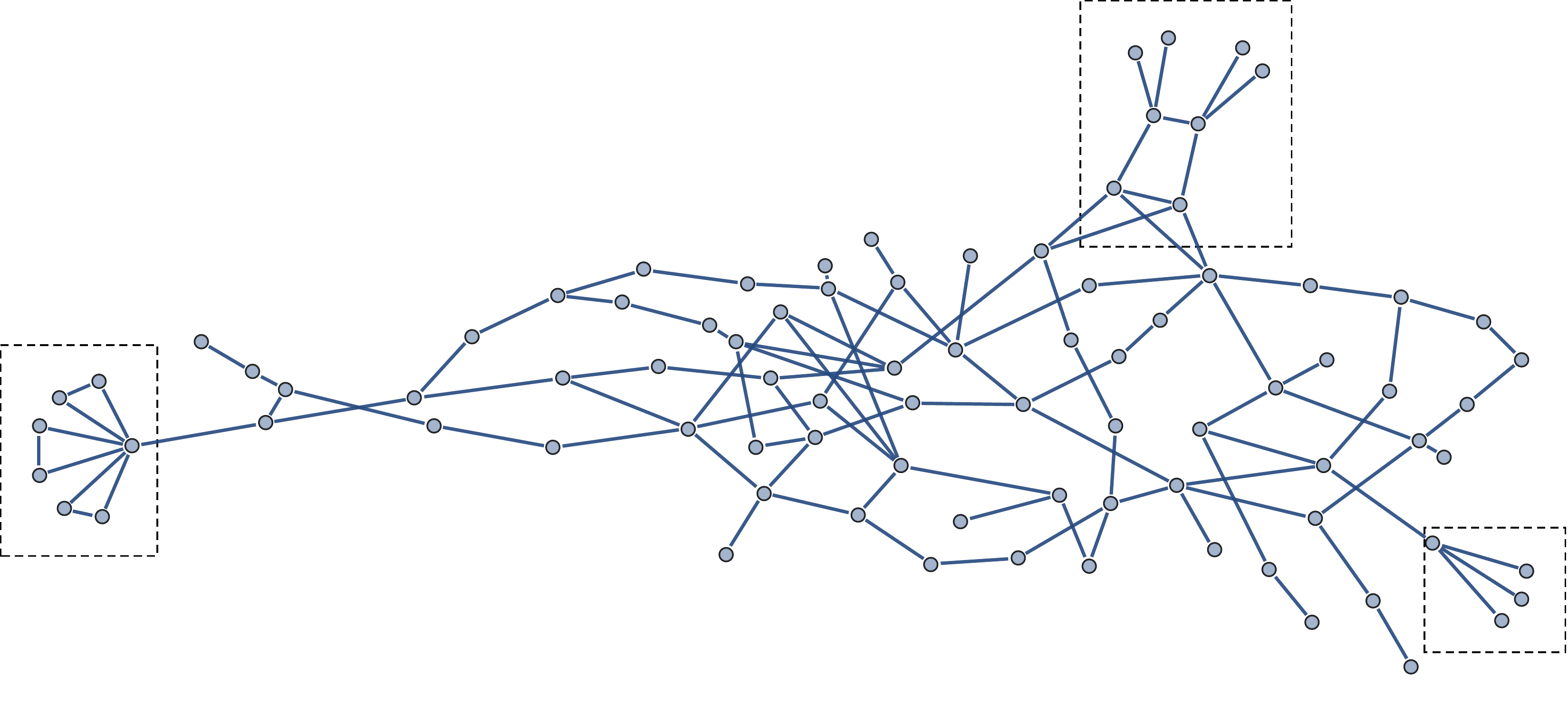}
    \put(71.5,42.25){\tiny$5$}
    \put(74.5,43.25){\tiny$7$}
    \put(79.25,42.5){\tiny$6$}
    \put(81.5,40.5){\tiny$8$}
    \put(71.75,37.5){\tiny$3$}
    \put(77.5,36.5){\tiny$4$}
    \put(69.5,33){\tiny$1$}
    \put(76.25,31.5){\tiny$2$}

    \put(6.25,10.25){\tiny$9$}
    \put(2.75,10.9){\tiny$10$}
    \put(0.4,14){\tiny$11$}
    \put(0.4,17.5){\tiny$12$}
    \put(1.75,20){\tiny$13$}
    \put(5.75,21.5){\tiny$14$}

    \put(96.5,4){\tiny$15$}
    \put(97.75,5.5){\tiny$16$}
    \put(98.25,7.75){\tiny$17$}

    \end{overpic}\end{center}
    
    \vspace{-0.25in}
    
\caption{The graph structure of a network on eighty-six vertices is shown. Three distinct graph symmetries are highlighted using boxes, which correspond to the automorphisms $\phi_1=(12)(34)(56)(78)$, $\phi_2=(9,11,13)(10,12,14)$, and $\phi_3=(15,16,17)$, respectively.}\label{fig:network}
\end{figure}
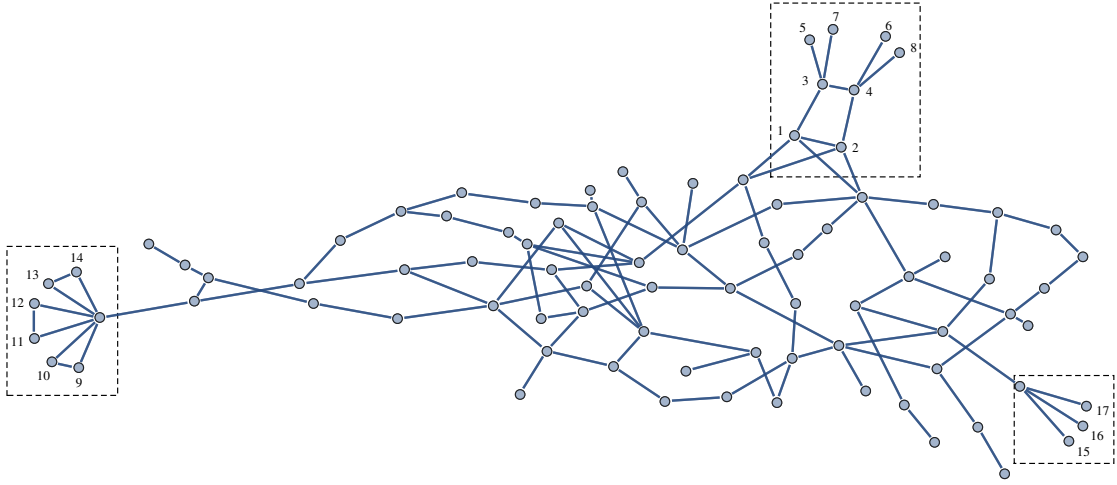

\begin{example}\label{ex:network}
Consider the undirected graph on eighty-six vertices in Figure \ref{fig:network}, which we think of as representing the graph structure of some network. Highlighted are three symmetries of the graph that correspond to the basic automorphisms
\[
\phi_1=(12)(34)(56)(78), \ \ \phi_2=(9,11,13)(10,12,14), \ \ \text{and} \ \ \phi_3=(15,16,17)
\]
respectively. Here, those vertices that are fixed by a given automorphism are omitted in this representation. We refer to these as the \emph{core} vertices of the network with respect to $\phi_i$ for $i=1,2,3$.

The idea is that by knowing a (local) graph symmetry, equivalently a (basic) automorphism, we can recover without much computational effort a number of the eigenvalues associated with the graph. Specifically, we can recover those eigenvalues associated with those vertices that are not fixed by the automorphism.

For instance, suppose we consider the $86\times 86$ adjacency matrix $A$ of the graph $G$. Using Theorem \ref{thm:2} we can decompose $A$ with respect to $\phi_1$. This gives the $82 \times 82$ matrix $B_{0,1}$ and the $4 \times 4$ matrix  $B_{1,1}$, where the second subscript indicates that we are decomposing the matrix $A$ with respect to $\phi_1$. Importantly, the matrix $B_{0,1}$ represents the core of the network fixed by $\phi_1$ together with half of those permuted by $\phi_1$, while $B_{1,1}$ represents the other half of the vertices permuted by $\phi_1$.

It is trivial to find the eigenvalues of $B_{1,1}$ as opposed to computing all eigenvalues of $A$. In fact, using Theorem \ref{thm:2} we quickly compute
\[
B_{1,1}=
\left[\begin{array}{cccc}
-1&1&0&0\\
1&-1&1&1\\
0&1&0&0\\
0&1&0&0
\end{array}\right],
\]
from which we find that $\sigma(B_{1,1})\approx \{1.170, 0, -.689, -2.481\}\subset\sigma(A)$.

Similarly, using $\phi_2$ we find that $A$ is similar to $B_{0,2}\oplus B_{{1,2}}\oplus B_{2,2}$ where $B_{0,2}\in\mathbb{N}^{82\times 82}$ and
\[
B_{1,2}=B_{2,2}=\left[\begin{array}{cc}
1&0\\
0&1
\end{array}\right],
\]
where we can quickly compute that $\sigma(B_{1,2})=\sigma(B_{2,2})=\{-1,1\}\subset \sigma(A)$. The same can be done for $\phi_3$ where we find that $A$ is similar to $B_{0,3}\oplus B_{1,3}\oplus B_{2,3}$ where $B_{1,3}=B_{2,3}=[0]$. Thus, $\sigma(B_{1,3}\oplus B_{2,3})=\{0,0\}\subset\sigma(A)$. Hence, knowing $\phi_i$ for $i=1,2,3$ we can quickly compute that
\[
\{-2.481, -.689, -1,-1,0,0,0,1,1,1.170\}\subset\sigma(A).
\]

Using the same procedure on the Laplacian matrix $L$ of $G$ we similarly find that
\[
\{0.523,1,1,1,1,1,3,3,3.552,5.925\}\subset\sigma(L).
\]
It is worth emphasizing that to compute these eigenvalues of $A$ and $L$, it does not matter how large the set of core vertices of a network are, only the local symmetries are needed to find them. It is also worth noting that, the much larger matrix $B_{0,i}$ is the divisor matrix of the equitable partition associated with $\phi_i$ for $i=1,2,3$. That is, the theory of equitable partitions does not provide a way of finding the eigenvalues of these smaller matrices that correspond to the symmetries of a graph, only those associated with its core.
\end{example}

An importance consequence of knowing some subset of a network's or graph's set of symmetries is that it gives us a way of quickly estimating properties related to dynamics processes of the associated network. For the sake of illustration, consider the network of $N$ coupled oscillators whose dynamics is governed by
\begin{equation}\label{eq:Kuramoto}
\dot{x}_i=\omega_i+\sum_{j\neq i}K_{ij}\sin(x_j-x_i), \ \ \text{for} \ \ i=1,\dots,N.
\end{equation}
Here, $x_i$ denotes the phase of the $i$th oscillator, $\omega_i$ its intrinsic frequency, and $K\in\mathbb{R}^{N\times N}$ is the weighted matrix that describe the network's interactions. This set of equations represents the Kuramoto-model, one of the most studied models for synchronization in network science \cite{Aceborn}.

The Kuromoto-model is known to  exhibit phase locking, where two oscillators $i$ and $j$ are \emph{locked} if $\dot{x}_i-\dot{x}_j=0$ for all $i,j=1,\dots,N$. The local dynamics of these phase-locked states depend on the eigenvalues of the \emph{linearization matrix} $J\in\mathbb{R}^{N\times N}$ given by $J_{ij}=\partial x_i/\partial x_j$. Importantly, these phase locked states are unstable if any eigenvalue of the matrix $J$ has a positive real part.

As can be seen from Equation \eqref{eq:Kuramoto}, if there are symmetries in the structure of interactions given by the coupling matrix $K$, then there are symmetries in the linearization matrix $J$. If these symmetries correspond to some basic automorphism $\phi$ then its possible to quickly determine some number of eigenvalues of $J$ irrespective of how large $N$ is. That is, if $J$ can be decomposed into the matrices $B_0,\dots,B_{k-1}$ where if any $B_i$ has an eigenvalue with a positive real part, the corresponding phased lock state cannot be stable.

This strategy not only works for the Kuramoto model but for any linearization of a continuous-time or discrete-time dynamical system about an equilibria. The method in the continuous-time case is analogous to what is described for the Kuramoto model, with slight modifications. For a discrete-time dynamical system an equilibria is unstable if the corresponding linearization has an eigenvalue whose magnitude is greater than 1 (i.e. its spectral radius is greater than 1). Hence, an equilibria is unstable if and only if when its linearization is decomposed into the matrices $B_0,\dots,B_{k-1}$, one of these matrices has a spectral radius greater than 1.

Not only can this method of equitable decompositions be used to estimate the spectral radius of a matrix but also its spectral gap. {The \emph{spectral gap} of a matrix is typically defined to be the difference between the moduli of the matrix' two largest eigenvalues. For a Laplacian matrix the matrix' largest nonzero eigenvalue is its spectral gap, which is the matrix' \emph{algebraic connectivity} or \emph{Fiedler value} if its associated graph is connected \cite{Kirkland}. The spectral gap of a Laplacian matrix} determines a number of dynamic properties on certain networks including synchronization thresholds and the rate of convergence to synchronization and consensus \cite{Pikovsky,Almendral,Atay}. For the stochastic matrix $P$, the spectral gap is related to mixing times of the associated Markov chain, and to first-passage times of random {walks} on the network \cite{Donetti}.


Using the automorphisms $\phi_1$, $\phi_2$, and $\phi_3$ to decompose the Laplacian matrix $L$ of the network shown in Figure \ref{fig:network} allows us to show that the spectral gap of this matrix, {which is also its algebraic connectivity in this case}, is less than or equal to $\lambda_2=0.523$ (see Example \ref{ex:network}). Again the point is that to get this bound we only need to know a few local symmetries of the graph.


\section{Conclusion}\label{sec:conclusion}

The theory introduced in this paper, which extends the theory of equitable partitions, allows one to decompose a matrix associated with a graph over a specific class of automorphisms. The result is a smaller number of matrices whose collective eigenvalues are the same as those associated with the original larger matrix. Because of this, an equitable decomposition can be viewed as a decomposition of a matrix over a particular set of symmetries that preserve the matrix' eigenvalues.

The particular types of automorphisms we consider in this paper are what we refer to as either \emph{uniform} automorphisms or \emph{basic} automorphisms. Uniform automorphisms are global in the sense that every graph vertex is permuted by a uniform automorphism. Basic automorphisms are typically local in that some number of the graph's vertices is fixed by this type of automorphism. In a following paper we describe how the results we present here can be used to equitably decompose a graph over \emph{any} of its automorphisms not only those that are either uniform or basic. The key idea is that for any automorphism $\phi$, some power of $\phi$ is a basic automorphism.

Importantly, one consequence of knowing a basic automorphism of a graph is that one can use this local information regarding the graph's structure to directly determine a number of eigenvalues associated with the graph. This method of using local symmetries to find eigenvalues of a graph is potentially significant in analyzing the spectral properties of real-world networks. One reason for this is that the size of these real-world networks can make it computationally expensive to determine their entire spectrum. Another is that real-world networks typically have a high degree of symmetry when compared, for instance, with randomly generated graphs \cite{MacArthur}. With this in mind, the method of equitable decompositions could be a practical tool for bounding the spectrum associated with a network, which can be used to determine certain properties related to the network's dynamics.

We note that the only restriction on equitably decomposing a graph is that the matrix associated with the graph be automorphism compatible. That is, the matrix needs to respect the structure of the graph's automorphisms (see Definition \ref{def:autocomp}). Automorphism compatible matrices include the graph's adjacency matrix, various Laplacian matrices, distance matrix, etc. This class includes a weighted graph's  adjacency matrix if the graph's weights also respect the structure of the graph's automorphisms. Hence, a large number of matrices that are typically associated with a graph can be decomposed over any of the graph's uniform or basic automorphisms.

Not only does this method of decomposition allow one to break up matrices along symmetries of an associated graph, it also can be used to bound the number of simple eigenvalues of certain graphs. Our main result in this direction is that, if an undirected graph has either a uniform or basic automorphism $\phi$, then the maximal number of simple eigenvalues of the graph can be bounded in terms of the orbit structure of $\phi$ (see Corollary \ref{cor:simpleevals} and \ref{cor:basicevals}). These results extend those of Petersdorf-Sachs \cite{PetersdorfSachs} to automorphism compatible matrices.

Regarding equitable decompositions, a question that one can ask is to what extent the process of equitably decomposing a graph can be reversed, i.e. to what extent can symmetries be introduced into a graph by \emph{equitably composing} a number of smaller graphs? Of course these graph would have to have the \emph{correct} form to be composed together, so there may be some natural restrictions on the type of symmetries that can be inserted into a graph in this way. From the spectral point of view, these restrictions may also impose conditions on the type of eigenvalues that can be added to a graph via such compositions.

An additional question is to what extent equitable decompositions can be extended to graphs without symmetries. That is, as not all equitable partitions correspond to automorphisms of a graph, is it possible to equitably decompose a matrix with respect to such partitions? One can even ask the more general question which is, around what type of structures is it possible to decompose a matrix associated with a graph. Understanding what kinds of structures, specifically local structures, around which it is possible to decompose a matrix, graph, or network would likewise allow one to estimate the spectrum of the respective matrix, graph, or network without having to compute all of its eigenvalues.

\bibliography{references}{}
\bibliographystyle{alpha} 

\end{document}